\documentclass[12pt, a4paper]{amsart}

\usepackage{amsfonts}
\usepackage{amssymb}
\usepackage[T1]{fontenc}
\usepackage{lmodern}
\usepackage{microtype}

\usepackage{color}
\usepackage{todonotes}

\usepackage{constants}

\usepackage{array}
\usepackage{esint}
\usepackage[alwaysadjust]{paralist}

\allowdisplaybreaks

\usepackage[utf8]{inputenc}

\newcommand{\abs}[1]{\mathopen\lvert#1\mathclose\rvert}
\newcommand{\bigabs}[1]{\bigl\lvert#1\bigr\rvert}

\newcommand{\norm}[1]{\mathopen\lVert#1\mathclose\rVert}

\newcommand{\floor}[1]{\lfloor#1\rfloor}
\newcommand{\Nset}{{\mathbb N}}
\newcommand{\Rset}{{\mathbb R}}

\newcommand{\Sp}{{\mathbb S}}
\newcommand{\st}{\;:\;}

\newcommand{\dif}{\,\mathrm{d}}

\theoremstyle{plain}
\newtheorem{proposition}{Proposition}[section]
\newtheorem{lemma}[proposition]{Lemma}
\newtheorem{theorem}{Theorem}

\newtheoremstyle{addendumstyle}{\topsep}{\topsep}{\itshape}{}{\bfseries}{.}{.5em plus 1pt minus 1pt}{#1 #2 to #3}
\theoremstyle{addendumstyle}

\theoremstyle{definition}
\newtheorem{definition}{Definition}[section]
\theoremstyle{remark}
\newtheorem{remark}{Remark}[section]

\newtheoremstyle{claimstyle}{\topsep}{\topsep}{}{}{\bfseries}{.}{.5em plus 1pt minus 1pt}{#1}
\theoremstyle{claimstyle}

\numberwithin{equation}{section}

\date{\today}

\title[Weak approximation by bounded Sobolev maps]{Weak approximation by bounded Sobolev maps with values into complete manifolds}

\author[P. Bousquet]{Pierre Bousquet}

\address{
Pierre Bousquet\hfill\break\indent
Universit{\'e} de Toulouse \hfill\break\indent
Institut de Math\'ematiques de Toulouse, UMR CNRS 5219\hfill\break\indent
Universit\'e Paul Sabatier Toulouse 3\hfill\break\indent 
118 Route de Narbonne\hfill\break\indent
31062 Toulouse Cedex 9\hfill\break\indent
France}

\author[A. C. Ponce]{Augusto C. Ponce}

\address{
Augusto C. Ponce\hfill\break\indent
Universit{\'e} catholique de Louvain\hfill\break\indent
Institut de Recherche en Math{\'e}matique et Physique\hfill\break\indent
Chemin du cyclotron 2, bte L7.01.02\hfill\break\indent
1348 Louvain-la-Neuve\hfill\break\indent
Belgium}

\author[J. Van Schaftingen]{Jean Van Schaftingen}

\address{
Jean Van Schaftingen\hfill\break\indent
Universit{\'e} catholique de Louvain\hfill\break\indent
Institut de Recherche en Math{\'e}matique et Physique\hfill\break\indent
Chemin du cyclotron 2, bte L7.01.02\hfill\break\indent
1348 Louvain-la-Neuve\hfill\break\indent
Belgium}

\begin{document}

\begin{abstract}
We have recently introduced the trimming property for a complete Riemannian manifold $N^{n}$ as a necessary and sufficient condition for bounded maps to be strongly dense in $W^{1, p}(B^m; N^{n})$ when $p \in \{1, \dotsc, m\}$.
We prove in this note that even under a weaker notion of approximation, namely the weak sequential convergence, the trimming property remains necessary for the approximation in terms of bounded maps. 
The argument involves the construction of a Sobolev map having infinitely many analytical singularities going to infinity.
\end{abstract}



\maketitle{}

\section{Introduction and statement of the result}
Given a Riemannian manifold \(N^n\) embedded in the Euclidean space \(\Rset^{\nu}\), the space \(W^{1, p} (B^m; N^n)\) of Sobolev maps from the unit ball \(B^m \subset \Rset^{m}\) to \(N^n\) can be defined by 
\begin{multline*}
 W^{1, p} (B^m; N^n) 
 = \bigl\{ u \in W^{1, p} (B^m; \Rset^\nu) \st \\
  u \in N^n \text{ almost everywhere in \(B^m\)} \bigr\}.
\end{multline*}
This space arises in some geometrical settings (harmonic maps) 
and physical models (liquid crystals, gauge theories, elasticity).

One question concerning these spaces is whether and how Sobolev maps can be approximated by smooth maps. 
Due to the nonconvex character of the target manifold \(N^n\), the usual convolution 
by a family of mollifiers fails in general. 
However, when the target manifold \(N^n\) is compact and \(p \ge m\), the class \(C^\infty(\overline{B^m}; N^n)\) is strongly dense in \(W^{1, p} (B^m; N^n)\): for every map \(u \in W^{1, p} (B^m; N^n)\), there exists a sequence of smooth maps \((u_\ell)_{\ell \in \Nset}\) in \(C^\infty (\overline{B^m}; N^n)\) that converges in measure to \(u\) on \(B^m\) and
\[
  \lim_{\ell \to \infty} \int_{B^{m}} \abs{D u_\ell - Du}^p = 0;
\]
see \cite{SchoenUhlenbeck}. 
When \(1 \le p < m\), this results holds if and only if the homotopy group of \(N^{n}\) of order \(\floor{p}\) (integer part of \(p\)) is trivial; see \cite{Bethuel1991,HangLin2003}.

In a recent work \cite{BousquetPonceVanSchaftingen}, we have considered the question of what happens when the target manifold \(N^n\) is not compact, but merely complete.
The starting observation is that the same homotopy assumption on \(N^{n}\) is necessary and sufficient for every map in \((W^{1, p} \cap L^\infty)(B^m; N^n)\) to be strongly approximated by smooth maps. 
Choosing a \emph{closed} embedding of \(N^n\) into \(\Rset^\nu\), the boundedness of a measurable map \(u : B^m \to N^n\) with respect to the Riemannian distance coincides with its boundedness as a map into \(\Rset^\nu\); by \(L^\infty (B^m; N^n)\) we mean the class of such maps.
Therefore the problem of strong approximation by smooth maps is reduced to the approximation of Sobolev maps by bounded Sobolev maps when \(p \le m\).

In \cite{BousquetPonceVanSchaftingen}, we prove the following

\begin{theorem}
  If\/ \(1 \le p \le m\) is not an integer, then \((W^{1, p} \cap L^\infty)(B^m; N^n)\) is strongly dense in \(W^{1, p}(B^{m}; N^{n})\). 
\end{theorem}

When \(p \le m\) is an integer, a new obstruction, this time of analytical nature, arises. 
Indeed, even in the case \(p = m\), where Sobolev maps cease to be continuous but still have 
vanishing mean oscillation (VMO), there exist maps \(u \in W^{1, p}(B^p; N^n)\) for some complete manifolds \(N^n\) that cannot be strongly approximated by bounded maps~\cite{HajlaszSchikorra2014,BousquetPonceVanSchaftingen}. 
A characteristic of those pathological target manifolds \(N^n\) is that their geometry degenerates at infinity, and the examples available can be realized as an \(m\)-dimensional infinite bottle in \(\Rset^{m+1}\) with a thin neck.

In order to identify the mechanism that is hidden in these examples, we have introduced the \emph{trimming property}:

\begin{definition}
\label{definitionExtensionProperty}
Given \(p \in \Nset \setminus \{0\}\), the manifold \(N^n\) satisfies the \emph{trimming property of dimension \(p\)} whenever there exists a constant \(C > 0\) such that every map \(f \in C^\infty(\partial B^p; N^n)\) that has a Sobolev extension \(u \in W^{1, p}(B^p; N^n)\) also has a smooth extension \(v \in C^\infty(\overline{B^p}; N^n)\) such that
\[
  \int_{B^p} \abs{Dv}^p \leq C \int_{B^p} \abs{Du}^p.
\]
\end{definition}

The use of \(C^{\infty}\) maps is not essential, and other classes like Lipschitz maps or continuous Sobolev maps (\(W^{1, p} \cap C^{0}\)) yield equivalent definitions of the trimming property;
see e.g.~Proposition~6.1 in \cite{BousquetPonceVanSchaftingen}.
The condition above allows one to characterize the target manifolds \(N^n\) for which
every map has a strong approximation by bounded Sobolev maps:

\begin{theorem}
\label{theoremIntegerTrimming}
For every \(p \in \{1, \ldots, m\}\), the set \((W^{1, p} \cap L^{\infty})(B^m; N^n)\) is strongly dense in \(W^{1, p}(B^m; N^n)\) if and only if \(N^n\) satisfies the trimming property of dimension \(p\).
\end{theorem}

The trimming property can be seen to be always satisfied when \(p = 1\) by taking as \(v\) a shortest geodesic joining the points \(f(-1)\) and \(f(1)\). 
We focus therefore on the case \(p \ge 2\) and the trimming property 
fails.
In this case, one may hope to approximate every map in \(W^{1, p}(B^m; N^n)\) by bounded maps using some weaker topology.
In this note, we address the question of whether this is true for the weak sequential approximation: 
given \(u \in W^{1, p} (B^m; N^n)\), to find a sequence of maps \((u_\ell)_{\ell \in \Nset}\) in \((W^{1, p} \cap L^{\infty})(B^m; N^n)\) which converges in measure to \(u\) and satisfies
\[
 \sup_{\ell \in \Nset} \int_{B^m} \abs{D u_\ell}^p < + \infty.
\]
This notion of convergence is also known as weak-bounded convergence \cite{White1988}.

An inspection of the explicit examples from \cite{HajlaszSchikorra2014,BousquetPonceVanSchaftingen} of maps  \(u \in W^{1, p} (B^p; N^n)\) that have no strong approximation by bounded maps shows that they have nevertheless a weak sequential approximation by bounded maps; see Remark~\ref{remarkSingularityFinite} below.
Using a more subtle construction that involves infinitely many analytical singularities, we prove that the trimming property is still necessary for the weak sequential approximation.

\begin{theorem}
\label{theoremMain}
Let \(N^n\) be a connected closed embedded Riemmanian manifold in \(\Rset^\nu\) and \(p \in \{2, \dotsc, m\}\). 
Every map \(u \in W^{1, p} (B^m; N^n)\) has a sequence in \((W^{1, p} \cap L^\infty)(B^m; N^n)\) that converges weakly to \(u\) if and only if \(N^n\) satisfies the trimming property of dimension \(p\).
\end{theorem}

This settles the question of weak sequential approximation by bounded Sobolev maps.
The counterpart of the weak sequential approximation by smooth maps is still open even for compact manifolds \(N^{n}\), except when \(p = 1\) and \(p = 2\), or when \(N^{n}\) is \((p-1)\)--connected, in which cases the weak sequential approximation always has an affirmative answer~\cite{Hang,PakzadRiviere2003,Hajlasz1994}; a recent counterexample by F.~Bethuel~\cite{Bethuel} gives a surprising negative answer in \(W^{1, 3}(B^{4}; \mathbb{S}^{2})\).{}
His proof involves a map with infinitely many \emph{topological} singularities, which are modeled after the Hopf map from \(\Sp^{3}\) onto \(\Sp^{2}\).

\section{Proof of the main result}

Our proof of Theorem~\ref{theoremMain} relies on a counterexample obtained by gluing 
together maps for which the trimming property degenerates. 
To perform this we first need to ensure that the trimming property fails on maps with a small Sobolev norm.

We shall say that a map \(u : B^{p} \to N^{n}\) is \emph{semi-homogeneous} if \(u(x) = f(x/\abs{x})\) for \(\abs{x} \ge 1/2\) and some function \(f : \partial B^{p} \to N^{n}\); by abuse of notation, we write \(f = u\).
The following property is a consequence of Lemma~6.4 in \cite{BousquetPonceVanSchaftingen}:

\begin{lemma}
\label{lemmaTrimmingSmall}
Let \(p \in \Nset\setminus\{0\}\), \(\beta > 0\) and \(C > 0\).
Assume that for every semi-homogeneous map \(u \in W^{1, p}(B^p; N^n)\) such that 
\[
  \int_{B^p} \abs{D u}^p \le \beta,
\]
there exists a map \(v\in (W^{1, p}\cap L^{\infty})(B^p; N^n)\) such that \(v = u\) on \(B^p \setminus B^p_{1/2}\) and 
\[
  \int_{B^p} \abs{D v}^p \le C \int_{B^p} \abs{Du}^p.
\]
Then, the manifold \(N^n\) has the trimming property.
\end{lemma}

\begin{proof}
We recall that Lemma~6.4 in \cite{BousquetPonceVanSchaftingen} asserts that the manifold \(N^n\) has the trimming property whenever there exist \(\alpha > 0\) and \(C' > 0\) such that for every map \(\tilde u \in W^{1, p}(B^p; N^n)\) (not necessarily semi-homogeneous) which satisfies \(\tilde u \in W^{1, p}(\partial B^p; N^n)\) and
\begin{equation}
  \label{eqTrimmingBoundedAlpha}
 \int_{B^p} \abs{D \tilde u}^p + \int_{\partial B^p} \abs{D \tilde u}^p \le \alpha,  
\end{equation}
 there exists a map \(\tilde v \in (W^{1, p}\cap L^{\infty})(B^p; N^n)\) such that \(\tilde v = \tilde u\) on \(\partial B^p\) and 
\begin{equation}
  \label{eqTrimmingBoundedAlphaConclusion}
  \int_{B^p} \abs{D \tilde v}^p 
  \le C' \biggl( \int_{B^p} \abs{D\tilde u}^p +  \int_{\partial B^p} \abs{D \tilde u}^p \biggr). 
\end{equation}
But in this case one can take \(u : B^{p} \to N^{n}\) defined by
\[{}
u(x)
= 
\begin{cases}
  \tilde u(2x)  & \text{if \(\abs{x} \le 1/2\),}\\
  \tilde u(x/\abs{x}) & \text{if \(\abs{x} > 1/2\),}
\end{cases}
\]
which is semi-homogeneous, also belongs to \(W^{1, p}(B^p; N^{n})\) and satisfies
\[{}
\resetconstant
\int_{B^{p}}{\abs{D u}^{p}}
\le \int_{B^{p}}{\abs{D \tilde u}^{p}} + \Cl{cte-297} \int_{\partial B^{p}}{\abs{D\tilde u}^{p}}
\le \max{\{1, \Cr{cte-297}\}} \alpha,
\]
where the constant \(\Cr{cte-297} > 0\) only depends on \(p\).{}

We thus have that, for every map \(\tilde u\) satisfying the estimate   \eqref{eqTrimmingBoundedAlpha} with \(\alpha = \beta/\max{\{1, \Cr{cte-297}\}}\), the semi-homogeneous map \(u\) defined above satisfies the assumptions of Lemma~\ref{lemmaTrimmingSmall}, and thus there exists a map \(v\) as in the statement.
The inequality \eqref{eqTrimmingBoundedAlphaConclusion} is thus verified with \(\Tilde{v} = v\) and \(C' = C \max{\{1, \Cr{cte-297}\}}\), hence by Lemma~6.4 in \cite{BousquetPonceVanSchaftingen} the manifold \(N^{n}\) satisfies the trimming property.
\end{proof}

We now quantify the behavior of the \(p\)-Dirichlet energy for the weak sequential approximation of a given map \(u\) in terms of  bounded extensions of \(u|_{\partial B^{p}}\), which are related to the trimming property.

\begin{lemma}
\label{lemmaSingularityRelaxedEnergy}
Let \(p \in \Nset\setminus\{0\}\). 
If \(u \in W^{1, p}(B^p; N^n)\) is semi-homogeneous and if \((u_\ell)_{\ell \in \Nset}\) is a sequence of maps in \((W^{1, p} \cap L^\infty) (B^p; N^n)\) that converges in measure to \(u\), then 
\begin{multline*}
  \liminf_{\ell \to \infty} \int_{B^p} \abs{D u_\ell}^p \ge c \inf \Bigl\{\int_{B^p} \abs{D v}^p \st v \in (W^{1, p} \cap L^\infty) (B^p; N^n) \\\text{ and \(u = v\) in \(B^p \setminus B^p_{1/2}\)} \Bigr\}.
\end{multline*}
\end{lemma}

\begin{proof}
We first recall that smooth maps are strongly dense in \((W^{1, p} \cap L^\infty)(B^p; N^n)\); see Proposition~3.2 in \cite{BousquetPonceVanSchaftingen}, in the spirit of \cite{SchoenUhlenbeck}.
By a diagonalization argument, we can thus assume  that, for every \(\ell \in \Nset\), \(u_\ell \in C^\infty (\overline{B^p}; N^n)\).
We may further restrict our attention to the case where the sequence \((Du_\ell)_{\ell\in \Nset}\) is bounded in \(L^{p}(B^p; \Rset^{\nu \times p})\). 
Since the sequence \((u_\ell)_{\ell \in \Nset}\) converges in measure to \(u\), it then follows that \((u_\ell)_{\ell \in \Nset}\) converges strongly to \(u\) in \(L^{p}(B^p; N^{n})\).{}

Indeed, the function \((|u_\ell-u|-1)_{+}\) vanishes on a set of measure greater than \(|B^p|/2\) for every \(\ell\) sufficiently large. 
Applying to this function the Sobolev--Poincar\'e inequality
\[{}
\resetconstant
\norm{f}_{L^{q}(B^{p})}
\le \C \norm{Df}_{L^{p}(B^{p})}
\]
valid for \(1 \le q < +\infty\) and \(f \in W^{1, p}(B^{p})\) that vanishes on a set of measure greater than  \(|B^p|/2\) (see p.~177 in \cite{Ziemer}), we deduce that \((u_{\ell})_{\ell \in \Nset}\) is bounded in \(L^{q}(B^p; N^{n})\) for \(q > p\).{}
The strong convergence of \((u_\ell)_{\ell \in \Nset}\) in \(L^{p}(B^m; N^{n})\) now follows from its convergence in measure and Vitali's convergence theorem for uniformly integrable sequences of functions.

By the integration formula in polar cooordinates and the Chebyshev inequality, there exists a radius \(r_\ell \in [1/2, 1]\) such that 
\begin{equation}
\label{eq303}
\int_{\partial B^p_{r_\ell}} \abs{D u_\ell}^p
\le 5 \int_{B^p \setminus B^p_{1/2}} \abs{D u_\ell}^p  \quad \text{and} \quad 
\int_{\partial B^p_{r_\ell}} \abs{u_\ell-u}^p
\le 5 \int_{B^p \setminus B^p_{1/2}} \abs{u_\ell-u}^p.
\end{equation}
We then define the map \(v_\ell: B^p \to \Rset^{\nu}\) for \(x \in B^p\) by 
\[
 v_\ell (x)=
 \begin{cases}
   u_\ell (4r_\ell x) & \text{if \(\abs{x} \le 1/4\)},\\
   (2 - 4 \abs{x}) u_\ell (r_\ell x /\abs{x}) + (4\abs{x}-1)u (x/\abs{x}) & \text{if \(1/4 \le \abs{x} \le 1/2\)},\\
   u (x) & \text{if \(\abs{x} \ge 1/2\)}.
 \end{cases}
\]
By the semi-homogeneity of \(u\), we have \(v_\ell \in (W^{1, p} \cap L^{\infty}) (B^p; \Rset^{\nu})\), and then by the estimates of Eq.~\eqref{eq303}, 
\[{}
\int_{B^p} \abs{D v_\ell}^p
\le \Cl{eq-330} \bigg(\int_{B^p}{\abs{D u_\ell}^p} + \int_{B^p \setminus B^{p}_{1/2}}{\abs{D u}^p}
+ \int_{B^p \setminus B^{p}_{1/2}}{|u_\ell-u|^p} \bigg).
\]
Since \((u_{\ell})_{\ell\in \Nset}\) is bounded in \(W^{1,p}(B^p;N^n)\) and converges in measure to \(u\), by weak lower semincontinuity of the \(L^{p}\) norm we obtain  
\[
\int_{B^p}\abs{D u}^p
\leq \liminf_{\ell\to \infty}\int_{B^p}\abs{D u_\ell}^p.
\]
Hence, by the strong convergence of \((u_{\ell})_{\ell \in \Nset}\) to \(u\) in \(L^{p}(B^p; N^{n})\), we get
\[
\liminf_{\ell\to \infty}\int_{B^p}\abs{D v_\ell}^p
\leq 2 \Cr{eq-330} \liminf_{\ell\to \infty}{\int_{B^p}\abs{D u_\ell}^p }.
\]

Recalling that the manifold \(N^{n}\) is embedded into \(\Rset^\nu\), there exists an open set \(U\supset N^{n}\) and a retraction \(\pi \in C^1 (U; N^n)\) such that \(\norm{D\pi}_{L^\infty (U)} \le 2\) and \(\pi \vert_{N^{n}} = \operatorname{id}\).
By the second inequality in \eqref{eq303}, the sequence \((u_\ell (r_\ell\ \cdot))_{\ell \in \Nset}\) converges to \(u(r_{\ell}\, \cdot) = u\) in \(L^{p}(\partial B^p; N^{n})\).
By the Morrey embedding, this convergence is also uniform on \(\partial B^{p}\), and \(u(\partial B^{p})\) is a compact subset of \(N^{n}\). 
Hence, for \(\ell\) large enough we have \(v_\ell (B^p) \subset U\), and we can thus take \(v = \pi \circ v_\ell\) to reach the conclusion; see e.g.~Lemma~2.2 in \cite{BousquetPonceVanSchaftingen}. 
\end{proof}

Two maps in \(W^{1, p}(B^{p}; N^{n})\) may not be glued to each other because they could have different boundary values.
We remedy to this problem by first gluing together two copies of the same map with reversed orientations, and then extending the resulting map to achieve a new map which is constant on \(B^{p} \setminus B^{p}_{1/2}\) (in particular semi-homogeneous), in the spirit of the dipole construction \cite{BrezisCoron,BrezisCoronLieb,Bethuel}.

\begin{lemma}
\label{lemmaPairing}
Let \(p\in \Nset\setminus \{0\}\). 
For every \(u\in W^{1,p}(B^p; N^n)\), there exists a map \(v\in W^{1,p}(B^p; N^n)\) such that 
\begin{enumerate}[\((i)\)]
\item \label{constantieie} the map \(v\) is constant on \(B^p \setminus B^{p}_{1/2}\),
\item for every \(x \in B_{1/8}^{p}\), \(v(x)=u(8x)\),
\item 
\[
\int_{B^p}\abs{Dv}^p \leq C \int_{B^p}|Du|^p.
\]
\end{enumerate}
\end{lemma}

\begin{proof}
We first apply the opening construction to the map \(u\) around \(0\); see Proposition~2.2 in \cite{BousquetPonceVanSchaftingen2015} (the opening technique has been introduced in \cite{BrezisLi}). This gives a map \(\tilde{u}\in W^{1,p}(B^p;N^n)\) such that \(\tilde{u}\) is constant on \(B_{1/2}^p\), agrees with \(u\) in a neighborhood of \(\partial B^{p}\) and satisfies \(\resetconstant \|D\tilde{u}\|_{L^{p}(B^p)}\leq \C \|Du\|_{L^{p}(B^p)}\).
We then introduce the map \(v : B^{p} \to N^{n}\) defined by
\[
v(x)=
  \begin{cases}
    u(8 x)                                       & \text{if \(\abs{x} < 1/8\),}\\
    \tilde{u}\bigl(16 x /(1+(8\abs{x})^2)\bigr)  & \text{if \(\abs{x} \ge 1/8\).}
  \end{cases}
\]
Since \(u=\tilde{u}\) on \(\partial B^p\) in the sense of traces,  the map \(v\) belongs to \(W^{1,p}(B^p;N^n)\). Moreover, the fact that 
\(\tilde{u}\) is constant on \(B^{p}_{1/2}\) implies that \(v\)  is constant on \({B^p\setminus B^{p}_{1/2}}\) and thus
\[
  \int_{B^p}\abs{Dv}^p 
  = \int_{B^{p}_{1/2}}\abs{Dv}^p 
  \leq  \int_{B^{p}}\abs{Du}^p 
    + \C \int_{B^{p}}\abs{D\tilde{u}}^p 
  \leq \C \int_{B^{p}}\abs{Du}^p.
\]
The proof is complete.
\end{proof}

Here is a geometric interpretation of the above proof: we consider two copies of \(u\) on the two hemispheres of the sphere \(\Sp^p\), which coincide on the equator. We then \emph{open} the map \(u\) in a neighborhood of the north pole.  Using a stereographic projection centered at the north pole which maps \(\Sp^p\) onto \(\Rset^p\), we then get a map defined on the whole \(\Rset^p\), which agrees with \(u\) on \(B^p\) and which is constant outside a larger ball. 
Then, by scaling, we obtain the desired map \(v\).

We now modify a Sobolev map which is constant near the boundary into a new map with a \emph{prescribed} constant value near the boundary.

\begin{lemma}
  \label{lemmaBasePoint}
  Let \(p \in \Nset\setminus\{0, 1\}\), \(N^{n}\) be a connected manifold, and \(y_{0} \in N^{n}\). 
  For every \(\epsilon > 0\) and every map \(u \in W^{1, p}(B^p; N^{n})\) that is constant on \(B^{p} \setminus B^{p}_{1/2}\), there exist \(\sigma \in (0, 1/2)\) and a smooth function \(\zeta : [\sigma, 1] \to N^{n}\) such that \(\zeta = y_{0}\) on \([1/2, 1]\) and the map \(w : B^{p} \to N^{n}\) defined by
  \[{}
  w(x) =
  \begin{cases}
    u(x/\sigma)     & \text{if \(\abs{x} < \sigma\),}\\
    \zeta(\abs{x})  & \text{if \(\abs{x} \ge \sigma\),}
  \end{cases}
  \]
  belongs to \(W^{1, p}(B^p; N^{n})\) and satisfies
  \[{}
  \int_{B^{p}} \abs{Dw}^{p}
  \le \int_{B^{p}} \abs{Du}^{p} + \epsilon.
  \]
\end{lemma}

\begin{proof}
 Let \(y_{1} \in N^{n}\) be such that \(u = y_{1}\) on \(B^{p} \setminus B^{p}_{1/2}\).
Since the manifold \(N^n\) is connected, there exists a smooth curve \(\gamma \in C^\infty ([0, 1]; N^n)\) such that \(\gamma (0) = y_0\) and \(\gamma (1) = y_1\).
For \(\sigma \in (0, 1/2)\) and \(x \in B^{p}\), we then define 
\[
 w_\sigma (x) 
 = 
 \begin{cases}
   u ({x}/{\sigma}) & \text{if \(\abs{x} < \sigma\)},\\
   \gamma \bigl({\log \abs{x}}/{\log \sigma}\bigr) & \text{if \(\abs{x} \ge \sigma\)}.
 \end{cases}
\]
We compute that 
\[{}
\resetconstant
\int_{B^p} \abs{D w_\sigma}^p
= \int_{B^p} \abs{D u}^p + \int_{B^p \setminus B^p_{\sigma}} \frac{\bigabs{\gamma' \bigl(\frac{\log \abs{x}}{\log \sigma}\bigr)}^p }{\abs{\log {\sigma}}^p  \abs{x}^p} \dif x
\le \int_{B^p} \abs{D u}^p + \frac{\C\norm{\gamma'}_{L^\infty}^p}{\abs{\log{\sigma}}^{p - 1}}.
\]
Since \(p > 1\), we can take \(\sigma < 1/2\) small enough so that the last term is smaller than \(\epsilon\).
\end{proof}

\begin{proof}[Proof of Theorem~\ref{theoremMain}]
If \(N^{n}\) satisfies the trimming property, then the sequence can be constructed to converge strongly to \(u\), and in particular weakly, by Theorem~\ref{theoremIntegerTrimming}.
To prove the direct implication, we assume by contradiction that the manifold \(N^n\) does not satisfy the trimming property,
and in this case we construct a map \(u \in W^{1, p} (B^m; N^n)\) that does not have a weak sequential approximation.

We first consider the case \(m = p\).
In view of Lemma~\ref{lemmaTrimmingSmall} with \(\beta = {1}/{2^j}\) and \(C=2^j\), where \(j \in \Nset\), by contradiction assumption there exists a semi-homogeneous map \(z_{j} \in W^{1, p} (B^p; N^{n})\) such that 
\[
 \int_{B^p} \abs{D z_j}^p \le \frac{1}{2^j},
\]
and for every \(v \in (W^{1, p}\cap L^{\infty})(B^p; N^n)\) such that
\(v = z_{j}\) on \(B^p \setminus B^p_{1/2}\), we have 
\begin{equation}
  \label{eqLowerBoundRenormalizedEnergy}
  \int_{B^p} \abs{D v}^p > 2^{j} \int_{B^p} \abs{Dz_{j}}^p.
\end{equation}
We then take an integer \(k_j \ge 1\) such that 
\begin{equation}
  \label{eqMultiplicitySingularities}
  \frac{1}{2^{j + 1}} 
  \le k_j \int_{B^p} \abs{D z_j}^p
  \le \frac{1}{2^j}.
\end{equation}

By Lemma~\ref{lemmaPairing}, we next replace \(z_{j}\) by a  map \(v_j\) which is constant on \(B^{p} \setminus B_{1/2}^{p}\), agrees with \(z_j(8x)\) for \(x\in B^{p}_{1/8}\), and satisfies 
\(\norm{Dv_j}_{L^{p}(B^{p})}^{p} \leq C \norm{Dz_{j}}_{L^{p}(B^{p})}^{p}\).
By Lemma~\ref{lemmaBasePoint}, we may further replace \(v_j\) by a  map \(w_{j}\) which equals some fixed value \(y_0 \in N^{n}\) on \(B^{p} \setminus B_{1/2}^{p}\), satisfies 
\[{}
w_j(x)
= v_j(x/\sigma_j){}
= z_{j}(8 x/\sigma_{j})
\quad \text{on \(B^{p}_{\sigma_j/8}\)}
\]
for some \(\sigma_j\in (0,\frac{1}{2})\), and is such that
\begin{equation}
    \label{eqEnergyRescaledMaps}
    \resetconstant
  \int_{B^{p}} \abs{Dw_{j}}^{p}
  \le \int_{B^{p}} \abs{Dv_j}^{p} + \epsilon_{j}
  \le C \int_{B^{p}} \abs{Dz_{j}}^{p} + \epsilon_{j}.
\end{equation}
Here \(\epsilon_{j} > 0\) is chosen so that the sequence \((k_{j}\epsilon_{j})_{j \in \Nset}\) is summable.

For \(j \in \Nset\) and \(i \in \{1, \dotsc, k_j\}\) we now choose points \(a_{i, j} \in \Rset^{p}\) and radii \(r_{i, j} > 0\) such that the balls \(B^p_{r_{i, j}} (a_{i, j})\) are contained in \(B^p_{1/2}\) and are mutually disjoint with respect to \(i\) and \(j\).
We then define the map \(u : B^p \to N^{n}\) by 
\[
 u (x)
 = \begin{cases}
    w_j \bigl(\tfrac{x - a_{i, j}}{r_{i, j}}\bigr) & \text{if \(x \in B^p_{r_{i, j}} (a_{i, j})\)},\\
    y_0 &  \text{otherwise}.
   \end{cases}
\]
Observe that \(u \in W^{1, p}(B^p; N^{n})\) and, by estimates \eqref{eqEnergyRescaledMaps} and \eqref{eqMultiplicitySingularities},
\[{}
\resetconstant{}
\begin{split}
 \int_{B^p} \abs{D u}^p
 & = \sum_{j \in \Nset} \sum_{i = 1}^{k_j} \int_{B^p_{r_{i, j}} (a_{i, j})} \bigabs{Dw_{j}\bigl(\tfrac{x - a_{i, j}}{r_{i, j}}\bigr)}^p \dif x\\
 & = \sum_{j \in \Nset} k_j \int_{B^p} \abs{D w_j}^p \\
 &\le \sum_{j \in \Nset} k_j\biggl( C \int_{B^{p}} \abs{Dz_{j}}^{p} + \epsilon_{j} \biggr) 
 \le 2 C + \sum_{j \in \Nset}{k_{j}\epsilon_{j}}.
\end{split}
\]

We now prove that \(u\) cannot be weakly approximated by a sequence of bounded Sobolev maps.
For this purpose, let \((u_\ell)_{\ell \in \Nset}\) be a sequence in \((W^{1, p} \cap L^{\infty})(B^p; N^{n})\) that converges in measure to \(u\).
For every \(j \in \Nset\) and \(i \in \{1, \dotsc, k_j\}\), it also converges in measure on the ball \(B^p_{r_{i, j}\sigma_j/8} (a_{i, j})\), where we have 
\[
u(x)=z_j \Big(\frac{8}{r_{i,j}\sigma_j}(x-a_{i,j})\Big).
\]
By the scaling invariance of the \(p\)-Dirichlet energy in \(\Rset^{p}\) and Eq.~\eqref{eqLowerBoundRenormalizedEnergy}, we deduce from Lemma~\ref{lemmaSingularityRelaxedEnergy} that
\[
\liminf_{\ell \to \infty}{\int_{B^p_{r_{i, j}} (a_{i, j})} \abs{D u_\ell}^p} 
 \ge c\, 2^{j} \int_{B^p} \abs{Dz_{j}}^p.
\]
Hence, by Fatou's lemma for sums of series and the first inequality in Eq.~\eqref{eqMultiplicitySingularities} we get
\[
\begin{split}
 \liminf_{\ell \to \infty} \int_{B^p} \abs{D u_\ell}^p 
 &\ge \sum_{j \in \Nset} \sum_{i = 1}^{k_j}{\liminf_{\ell \to \infty} \int_{B^p_{r_{i, j}} (a_{i, j})} \abs{D u_\ell}^p }\\
 &\ge c \sum_{j \in \Nset}{k_j 2^{j} \int_{B^p} \abs{D z_j}^p}
 \ge c \sum_{j \in \Nset}{\frac{1}{2}} = + \infty,
\end{split}
\]
which yields a contradiction and completes the proof when \(p = m\).

When \(p < m\), we construct a counterexample \(u \in W^{1, p}(B^{m}; N^{n})\) to the weak sequential density that satisfies the additional property to be constant on \(B^{m} \setminus B^{m}_{1/2}\).
We proceed as follows: we first take a counterexample \(\Tilde{u} \in W^{1, p} (B^p; N^n)\) which equals \(y_{0}\in N^{n}\) on \(B^{p} \setminus B^{p}_{1/2}\). 
We then set, for \(x = (x', x'') \in \Rset^{p - 1} \times \Rset\),
\[
\bar u (x) = 
 \begin{cases}
   \Tilde{u} (2 x', 2 x'' - 1)
   & \text{if \(\abs{x'}^2 + (x'' -1/2)^2 \le 1/16\),}\\
   y_0 & \text{otherwise.}
 \end{cases}
\]
This map \(\bar u\), which is obtained from \(\Tilde{u}\) by translation and dilation, is still a legitimate counterexample in \(\Rset^{p}\).{}

We next perform a rotation of the upper-half subspace \(\Rset^{p-1} \times (0, \infty) \times \{0_{m-p}\} \subset \Rset^{m} \) around the axis \(\Rset^{p-1} \times \{0_{m-p+1}\}\) to define the desired map \(u : \Rset^{m} \to N^{n}\) for \(x = (x', x'') \in \Rset^{p-1} \times \Rset^{m - p + 1}\) by
\[{}
u(x) = \bar u(x', \abs{x''}).
\]
Geometrically, on every slice of the set 
\[{}
\bigl\{x \in \Rset^{m} \st \abs{x'}^2 + (\abs{x''} -1/2)^2 \le 1/16 \bigr\}
\] 
by a \(p\)-dimensional plane containing \(\Rset^{p-1} \times \{0_{m-p+1}\}\), \(u\) is obtained by gluing two copies of \(\Tilde u\) on balls of radii \(1/2\), having opposite orientations.

We conclude by observing that \(u \in W^{1, p}(B^{m}; N^{n})\) is also a counterexample to the weak sequential density.
Indeed, if there were some weak sequential approximation of \(u\) by bounded Sobolev maps, then, by slicing \(\Rset^{m}\) with \(p\)-dimensional planes containing \(\Rset^{p-1} \times \{0_{m-p+1}\}\) and using a Fubini-type argument, we would have a weak sequential approximation of \(\Tilde{u}\) by bounded Sobolev maps in dimension \(p\).{}
By the choice of \(\Tilde{u}\), this is not possible.{}
\end{proof}

\begin{remark}
  \label{remarkSingularityFinite}
  The map provided by the counterexample above has infinitely many singularities.
  In the presence of only \emph{finitely many analytical} singularities, one shows that the problem of weak approximation by a sequence of bounded Sobolev maps has an affirmative answer.
  We justify below this observation in the case of one singularity at \(a = 0\) and for a map \(u \in C^{\infty}(\overline{B^{p}} \setminus \{0\}; N^{n})\) whose restriction \(u|_{\partial B^{p}}\) is homotopic to a constant in \(C^{0}(\partial B^{p}; N^{n})\);
  the latter property is always satisfied when the homotopy group \(\pi_{p-1}(N^{n})\) is trivial. 

  We proceed as follows: given a sequence of numbers \((r_{\ell})_{\ell \in \Nset}\) in \((0, 1)\) that converges to \(0\), for each \(\ell \in \Nset\) take the map \(u_{\ell} : \overline{B^{p}} \to N^{n}\) defined by
  \[{}
  u_{\ell}(x) =
  \begin{cases}
    u(x)              & \text{if \(\abs{x} \ge r_{\ell}\),}\\
    u(r_{\ell}^{2} x/\abs{x}^{2})   & \text{if \(r_{\ell}^{2} \le \abs{x} < r_{\ell}\),}\\
    H(\abs{x}/r_{\ell}^{2}, x/\abs{x}) & \text{if \(\abs{x} < r_{\ell}^{2}\),}
  \end{cases}
  \]
  where \(H : [0, 1] \times \partial B^{p} \to N^{n}\) is a smooth homotopy such that \(H(1, \cdot) = u|_{\partial B^{p}}\) and \(H(t, \cdot) = y_{0}\) for \(t \le 1/2\), and \(y_{0} \in N^{n}\) is  some given point.{}
  Observe that such a map \(u_{\ell}\) is defined through a Kelvin transform which sends the annulus \(B^{p}_{r_{\ell}} \setminus B^{p}_{r_{\ell}^{2}}\) onto  \(\overline{B^{p}_{1}} \setminus \overline{B^{p}_{r_{\ell}}}\) and preserves the \(p\)-Dirichlet energy.
  One then verifies that
  \[{}
  \int_{B^{p}}{\abs{Du_{\ell}}^{p}}
  \le 2 \int_{B^{p}}{\abs{Du}^{p}} + C \norm{DH}_{L^{\infty}([0, 1] \times \partial B^{p})}^{p}.
  \]  
  Since \(u_{\ell} = u\) on \(B^{p} \setminus B^{p}_{r_{\ell}}\), the sequence \((u_{\ell})_{\ell \in \Nset}\) converges in measure to the map \(u\).
\end{remark}

\section*{Acknowledgements}
The second author (ACP) warmly thanks the Institut de Math\'ematiques de Toulouse for the hospitality.
The third author (JVS) was supported by the Fonds de la Recherche scientifique--FNRS (Mandat d'Impulsion scientifique (MIS) F.452317).

\end{document}